\documentclass[letterpaper, 10 pt, conference]{cssconf}
\IEEEoverridecommandlockouts        
\usepackage{times}

\usepackage{graphicx}
\usepackage{amsmath}
\usepackage{amsfonts}
\usepackage{longtable,tabularx}
\usepackage{mathtools}
\DeclarePairedDelimiter{\norm}{\lVert}{\rVert}

\usepackage{amssymb}

\newtheorem{thm}{Theorem}
\newtheorem{lmm}{Lemma}

\newtheorem{ass}{Assumption}

\usepackage{threeparttable}

\newenvironment{smallarray}[1]
 {\null\,\vcenter\bgroup\scriptsize
 \arraycolsep=.13885em
 \hbox\bgroup$\array{@{}#1@{}}}
 {\endarray$\egroup\egroup\,\null}

\usepackage{microtype}

\begin{document}
\title{Robustness and Convergence Analysis of First-Order Distributed Optimization Algorithms over Subspace Constraints}

\author{Dennis~J.~Marquis, 
  Dany~Abou~Jaoude, 
  Mazen~Farhood, 
  and~Craig~A.~Woolsey
\thanks{D. Marquis, M. Farhood, and C. Woolsey are with the Kevin T. Crofton Department of Aerospace and Ocean Engineering, Virginia Tech, Blacksburg, VA 24061, USA (e-mail: \{dennisjm, farhood, cwoolsey\}@vt.edu). D. Abou Jaoude is with the Department of Mechanical Engineering, American University of Beirut, Lebanon (e-mail: da107@aub.edu.lb).}\thanks{This work was supported by the Office of Naval Research (ONR) under Award No. N00014-18-1-2627 and the Army Research Office (ARO) under Contract No. W911NF-21-1-0250.}}

\markboth{Journal of \LaTeX\ Class Files,~Vol.~14, No.~8, August~2015}%
{Shell \MakeLowercase{\textit{et al.}}: Bare Demo of IEEEtran.cls for IEEE Journals}
%

\maketitle

\begin{abstract}
This paper extends algorithms that remove the fixed point bias of decentralized gradient descent to solve the more general problem of distributed optimization over subspace constraints. Leveraging the integral quadratic constraint framework, we analyze the performance of these generalized algorithms in terms of worst-case robustness and convergence rate. The utility of our framework is demonstrated by showing how one of the extended algorithms, originally designed for consensus, is now able to solve a multitask inference problem. 
\end{abstract}

%
\IEEEpeerreviewmaketitle

\section{Introduction}
\label{sec:intro}

There is considerable literature on iterative optimization algorithms that rely on local information exchange to solve the consensus problem. Decentralized gradient descent (DGD) has been shown to exhibit linear convergence to a fixed point when minimizing strongly convex objective functions with a fixed step-size $\mu$~\cite{Yuan2016}. Unlike its centralized counterpart, however, DGD converges to a $\textit{O}(\mu)$ neighborhood of the optimal solution. Diminishing step-size schemes allow DGD to converge without a bias, but the resulting convergence rate is sub-linear. Algorithms have been developed capable of achieving linear convergence rates for strongly convex objective functions without this fixed point bias, including EXTRA~\cite{Shi2015}, NIDs~\cite{Li2019}, Exact Diffusion~\cite{Yuan2019}, DIGing~\cite{Nedic2017}, unified methods~\cite{Jakovetic2019}, AugDGM~\cite{Xu2015}, and SVL~\cite{Sundararajan2020}. Some of these algorithms offer additional improvements, such as the ability to handle time-varying graphs or allowing uncoordinated step-sizes amongst agents. Techniques used to accelerate centralized algorithms have also been applied to the distributed setting. For example, Nesterov acceleration~\cite{Nesterov2018} has been used to improve distributed algorithm convergence rates when the objective function has a high condition ratio~\cite{Qu2020}. Recently, these bias-correction algorithms have started to be analyzed in the stochastic setting~\cite{Yuan2020}, specifically to determine under which scenarios they outperform traditional methods like DGD.

In addition to algorithms designed exclusively for the consensus problem, there are algorithms that solve a more general class of problems: distributed optimization over subspace constraints. In \cite{Nassif2020a}, it is shown that many common distributed optimization problems can be cast as distributed subspace constrained optimization problems, including consensus optimization, coupled optimization, optimization under affine constraints, and band-limited graph signal estimation. Two iterative algorithms that use local information exchange to solve this type of problems are DiSPO~\cite{Lorenzo2019} and the distributed adaptive strategy proposed in \cite{Nassif2020a,Nassif2020b}, which will be referred to as DAS in the sequel. These two algorithms have update equations that are analogous to DGD and a diffusion form~\cite{Sayed2014} of DGD, respectively. Consequently, DiSPO and DAS also exhibit a fixed point~bias.

The integral quadratic constraint (IQC) framework~\cite{Megretski1997SystemConstraints} provides an approach to analyze the stability of dynamical systems by modeling these systems as a feedback interconnection between a linear time-invariant (LTI) system $G$ and an uncertainty operator $\Delta$ that lies in a pre-specified set $\mathbf{\Delta}$, described by an IQC. The original analysis conditions in the IQC framework were in the frequency domain. More recent works that rely on dissipativity theory have proposed IQC analysis conditions derived in the time domain; see, e.g., \cite{Jaoude2019,Fry2021}, for relevant works in discrete time. 

In~\cite{Lessard2016}, the IQC framework is adapted to analyze the convergence rate of optimization algorithms. In this setup, optimization algorithms are interpreted as discrete time dynamical systems and the operator $\Delta$ represents the nonlinear gradient computations. Analysis results are formulated as semidefinite programs (SDPs), therefore this approach is constructive and does not rely on algorithm-specific expert knowledge to produce algorithm performance guarantees. 
The IQC framework has also been used in algorithm design~\cite{VonScoy2018,Cyrus2018}.

In this paper, we extend the consensus algorithms that remove the fixed point bias of DGD to solve the more general problem of distributed optimization over subspace constraints. By modifying the weights of local information exchange, these algorithms can now be applied to new problems, such as multitask inference. Additionally, we propose an IQC-based framework to analyze the performance of these generalized algorithms in terms of worst-case robustness and convergence rate. Our analysis framework is an extension of those in \cite{Sundararajan2020,Sundararajan2018}, which are used to analyze the convergence rate of first-order distributed consensus optimization algorithms, and the framework in~\cite{VanScoy2021}, which is used to analyze the robustness of first-order centralized optimization algorithms. In addition to proposing extended algorithms for handling subspace constraints, our contributions consist of (1) handling a more general class of distributed optimization problems and (2) extending the robustness analysis results to the distributed setting. As an illustrative case study, our framework is used to compare the extended version of AugDGM to DAS for solving a multitask inference problem.

This paper is organized as follows. In Section~\ref{sec:prelim}, the preliminaries are introduced. In Section~\ref{sec:general}, new algorithms are proposed for distributed optimization under arbitrary subspace constraints. Section~\ref{sec:analysis} presents the analysis results. Section~\ref{sec:bias} presents the case study. Section~\ref{sec:conclude} concludes~the~paper.
\section{Preliminaries}
\label{sec:prelim}

\subsection{Notation}
$\mathbb{R}^n$ denotes the space of real-valued vectors of dimension $n$. $\mathbb{N}$ corresponds to the set of natural numbers. $X \succ 0$ and $X \succeq 0$ indicate that a symmetric real matrix $X$ is positive definite and positive semidefinite, respectively. $\mathcal{R}(M)$, $\mathcal{N}(M)$, and $\mathrm{tr}(M)$ denote the range, nullspace, and trace of matrix $M$, respectively. $\mathbf{0}$ denotes a zero matrix of appropriate dimension. $I_i$ denotes the $i \times i$ identity matrix. $\mathbf{1}_N$ is an N-entry vector of ones. The symbol $\otimes$ denotes the Kronecker product. $\mathrm{col}\{v_k\}_{k=1}^N$ denotes the vertical concatenation of the vectors $v_1,\ldots,v_N$ and $\mathrm{diag}\{\lambda_k\}_{k=1}^N$ denotes  $N \times N$ matrix formed from the diagonal augmentation of the scalars $\lambda_1,\ldots,\lambda_N$. Given a projection matrix $\mathcal{P_U}$ and a network with gossip matrix $\mathcal{A}$, we define the spectral gap of matrix $\mathcal{A}$ as $\sigma \coloneqq \norm{\mathcal{A}-\mathcal{P_U}}$, where $0\leq \sigma < 1$, $\mathcal{P_U}$ is a projection matrix, and $\norm{.}$ denotes the spectral norm. For a given vector $v$, $\norm{v}$ denotes the $\ell^2$ norm of $v$.

\subsection{Problem Formulation}
\label{sec:problem}
We consider a network consisting of $N$ agents, connected over an undirected graph $\mathcal{G}$. The set of vertices of $\mathcal{G}$ is defined as $\mathcal{V}=\{1,\ldots,N\}$, where each vertex $k\in\mathcal{V}$  corresponds to an agent, referred to as ``agent $k$.''  
Ordered pair $(i,j)$ is in edge set $\mathcal{E}$ if and only if there is an edge between vertices $i\in \mathcal{V}$ and  $j \in \mathcal{V}$. The  optimization problem to be solved is\vspace{-2mm}
\begin{equation}\label{eq:Jglobal}
\mathrm{min}\;\sum_{k=1}^NJ_k(\omega_k) \quad \mathrm{subject\;to}\quad \omega \in \mathcal{R}(\mathcal{U}),
\vspace{-2mm}\end{equation}
where $\mathcal{U}$ is a matrix with full column rank, whose columns form a basis of the subspace constraining $\omega \coloneqq \mathrm{col}\{\omega_k\}_{k=1}^N$.
Each agent $k$, for $k\in\mathcal{V}$, has access to part of the objective function, i.e. a local objective function $J_k:\mathbb{R}^d\rightarrow\mathbb{R}$ that is convex and continuously differentiable. The subspace constraint is a coupling constraint, otherwise, minimization of (\ref{eq:Jglobal}) would simply require minimizing each objective function locally.

The local objective functions satisfy the assumptions below.
\begin{ass} The local objective function $J_k$ of each agent $k$ has an $L_k$-Lipschitz continuous gradient; i.e.,
\begin{equation*}
\norm{\nabla J_k(\omega_a)-\nabla J_k(\omega_b)} \leq L_k\norm{\omega_a-\omega_b} \quad \mbox{for all $\omega_a,\omega_b \in \mathbb{R}^d$}.
\end{equation*}\label{ass:Lip}\vspace{-4mm}
\end{ass}
\begin{ass} The local objective function $J_k$ of each agent $k$ is $m_k$-strongly convex; i.e., for all $\omega_a,\omega_b \in \mathbb{R}^d$,
\begin{equation*}
J_k(\omega_b)\geq J_k(\omega_a) + \nabla J_k(\omega_a)^T(\omega_b-\omega_a)+\frac{m_k}{2}\norm{\omega_b-\omega_a}^2.
\end{equation*}\label{ass:Strong}\vspace{-4mm}
\end{ass}

Strong convexity of the local objective functions ensures that the global objective function is also strongly convex, so that the minimizer $\omega^{\mathrm{opt}}$ of (\ref{eq:Jglobal}) is unique. The condition ratio of the global objective function is defined as $\kappa\coloneqq\frac{L}{m}$, where $L\coloneqq\mathrm{max}\{L_k\}_{k=1}^N$ and $m\coloneqq\mathrm{min}\{m_k\}_{k=1}^N$. Define $\mathcal{S}(m_k,L_k)$ as the set of local objective functions that satisfy Assumption~\ref{ass:Lip} and Assumption~\ref{ass:Strong} for some $L_k$ and some $m_k$, respectively.
\subsection{Algorithm Form}
\label{sec:algform}
The first-order distributed algorithms to be analyzed can be modeled using the following state space representation:
\begin{subequations}\label{eq:algmodel}
\begin{gather}
\label{eq:algmodel_lin1}\xi^{t+1}=A\xi^t+B(u^t+w^t),\\
 \label{eq:algmodel_lin2}y^t=C_y\xi^t,\quad \omega^t=C_\omega\xi^t,\\
 \label{eq:algmodel_u} u^t=\Delta(y^t)=\mathrm{col}\{\nabla J_k(y_k^t)\}_{k=1}^N.
\end{gather}
\end{subequations}
The vector $\xi^t \in \mathbb{R}^{nn_{\mathrm{alg}}}$ is the state value at iteration $t$, where $n \coloneqq Nd$ and $n_{\mathrm{alg}}$ is specific to each algorithm. The input vector $u^t$ is a stack of the local gradients evaluated at the respective $y^t_k$ of each agent, $y^t\coloneqq \mathrm{col}\{y_k^t\}_{k=1}^N$. For each state vector $\xi^t$, the algorithm iterate $\omega^t$ can be measured using $C_\omega$. In the absence of gradient noise, $w^t \equiv 0$, each algorithm is assumed to have a fixed point $(\xi^*,y^*,u^*,\omega^*)$ that satisfies~(\ref{eq:algmodel}).
 
The dimension $d$ is the same for all agents and the algorithms to be analyzed will have state space matrices that exhibit a special structure: $A=\bar{A}\otimes I_d$ for some matrix $\bar{A}$, $B=\bar{B}\otimes I_d$ for some matrix $\bar{B}$, etc. This observation supports the following assumption, without loss of generality.
\begin{ass} 
The dimension $d=1$.
\end{ass}

The state update is corrupted by zero-mean additive noise $w^t$. Additive noise can be used to model gradient noise due to numerical errors or approximations. In learning applications, where the true objective function is unknown but approximated through sample collection, gradient noise models the error between the true and estimated objectives. The noise satisfies the following assumption.
\begin{ass} The noise $w^t$ is zero-mean (i.e. unbiased), with covariance $\mathbb{E}w^t(w^t)^T \preceq R$ for some $R\succ 0$, for all $t \in \mathbb{N}$. The noise sequence has joint distribution $\mathbb{P}$, which is independent across iterations. If $w \sim \mathbb{P}$, then $w^t$ and $w^\tau$ are independent for all $t \neq \tau$. 
\label{ass:noise}\end{ass}

Each algorithm satisfies a respective invariant:\vspace{-2mm}
\begin{equation}\label{eq:alginvariant}
F_\xi\xi^t+F_uu^t=0 \quad \mathrm{for\;all\;} t \in \mathbb{N}.
\vspace{-2mm}\end{equation}
An invariant follows from the initialization constraint of a given algorithm and is often necessary to produce a feasible solution using the upcoming analysis methods.

\subsection{Algorithm Performance Metrics}
The below performance metrics are adopted from \cite{VanScoy2021}.
\subsubsection{Rate of Convergence}
The performance metric $\rho$ is an upper bound on the worst-case linear convergence rate of an algorithm across all objective functions $f$ in a constraint set $\mathcal{S}$ and all initial conditions. This rate describes the transient phase of algorithm iterates, where the noise input $w^t$ is negligible compared to the gradient input $u^t$. Thus, for this performance metric, it is assumed there is no noise and a $\rho$ is computed such that $\norm{\xi^t-\xi^*}\leq c\rho^t\norm{\xi^0-\xi^*}$  for some constant $c>0$, for all $t\in \mathbb{N}$. Formally, the rate of convergence is defined as 
\begin{equation}\label{eq:rho}
\rho \coloneqq \mathrm{inf}\left\{\rho>0\biggm| \mathrm{sup} \; \frac{\norm{\xi^t-\xi^*}}{\rho^t\norm{\xi^0-\xi^*}}<\infty\right\}.
\end{equation}

\subsubsection{Sensitivity}
The sensitivity $\gamma$ is a measure of robustness to additive noise. It bounds the standard deviation of iterates produced by the algorithm during the steady state phase. The quantity $\gamma^2$ can be interpreted as a bound on the generalized $H_2$ norm of the system. Formally, the sensitivity is defined as
\begin{equation}\label{eq:gamma}
\gamma \coloneqq \underset{T \rightarrow \infty}{\mathrm{lim \; sup}} \sqrt{\mathbb{E}_{w\sim \mathbb{P}}\frac{1}{T}\sum_{t=0}^{T-1}\norm{\omega^t-\omega^*}^2}.
\end{equation}
\subsection{IQCs Describing Gradients of Convex Functions}
For first-order algorithm analysis, the nominal system $G$ corresponds to linear dynamics (\ref{eq:algmodel_lin1})-(\ref{eq:algmodel_lin2}) and uncertainty operator $\Delta$ corresponds to the local gradient computations performed in (\ref{eq:algmodel_u}). The IQC characterizing $\Delta$ can be defined by symmetric matrix $M$ and $z$, the output of operator $\Psi$, driven by $u$ and $y$. Specifically, $\Psi$ is defined as
\begin{subequations}\label{eq:psimodel}
\begin{gather}
\psi^{t+1}=A_\Psi\psi^t+B^y_\Psi y^t+B^u_\Psi u^t, \quad \psi^0=\psi^*, \\
\label{eq:psimodelz} z^t=C_\Psi\psi^t+D^y_\Psi y^t+ D^u_\Psi u^t,
\end{gather}
\end{subequations}
with a fixed point defined by $(\psi^*,y^*,u^*,z^*)$.

Suppose $u=\Delta(y)$, $z=\Psi(y,u)$, and $z^*=\Psi(y^*,u^*)$. The operator $\Delta$ is said to satisfy the (a) Pointwise, (b) $\rho$-Hard, (c) Hard, (d) Soft IQC defined by $(\Psi,M)$ if the respective inequality holds for all $y \in \mathbb{R}^n$ and $T\in \mathbb{N}$:
\begin{subequations}\label{eq:IQCtypes}
\begin{gather}
(z^t-z^*)^TM(z^t-z^*)\geq 0 \quad \mathrm{for\;all\;} t\in \mathbb{N},\\
\textstyle\sum_{t=0}^T\rho^{-2t}(z^t-z^*)^TM(z^t-z^*)\geq 0,\\
\textstyle\sum_{t=0}^T(z^t-z^*)^TM(z^t-z^*)\geq 0,\\
\label{eq:IQCtypes_soft}\textstyle\sum_{t=0}^\infty(z^t-z^*)^TM(z^t-z^*)\geq 0,
\end{gather}
\end{subequations}
where the summation in (\ref{eq:IQCtypes_soft}) is finite.

These four types of IQCs, introduced in~\cite{Lessard2016}, form nested sets, i.e., $\{\mathrm{pointwise\;IQCs}\} \subseteq \{\rho\mathrm{-hard\;IQCs},\;\rho<1\}\subseteq \{\mathrm{hard\;IQCs}\}\subseteq\{\mathrm{soft\;IQCs}\}$. 

The next IQCs characterize  objective functions of interest, where $\bar{m}\coloneqq \mathrm{diag}\{m_k\}_{k=1}^N$ and $ \bar{L}\coloneqq \mathrm{diag}\{L_k\}_{k=1}^N$.
\vskip 1mm
\begin{lmm}[Distributed Sector IQC]\label{lem:sector}
Given~$J_k{\in} \mathcal{S}(\!m_k,\!L_k)$ for all $k \ {\in} \ \mathcal{V}$, if $u=\Delta(y)$, then $\Delta$ satisfies the pointwise IQC defined by $\Psi=\left[\begin{smallmatrix} \bar{L} && -I_N \\ - \bar{m} && I_N\end{smallmatrix}\right]$ and $M=\left[\begin{smallmatrix} \mathbf{0} && I_N \\ I_N && \mathbf{0}\end{smallmatrix}\right]$.
\end{lmm}\vskip 1mm
\begin{proof} See~\cite{Lessard2016}, adapted for problems in the form (\ref{eq:Jglobal}). \end{proof}
\vskip 1mm
\begin{lmm}[Distributed Weighted Off-By-One IQC]\label{lem:weighted}
Given $J_k\in \mathcal{S}(m_k,L_k)$ for all $k \in \mathcal{V}$, if $u=\Delta(y)$, then for any ($\bar{\rho},\rho$) where $0\leq \bar{\rho}\leq\rho\leq 1$, $\Delta$ satisfies the $\rho$-hard IQC defined by $\Psi=\left[\begin{smallarray}{c|cc}\rule{0mm}{3mm}\mathbf{0} & - \bar{L} & I_N \\ \hline \rule{0mm}{3mm}\bar{\rho}^2I_N &  \bar{L} & -I_N \\ \mathbf{0} & - \bar{m} & I_N\end{smallarray}\right]$ and  $M=\left[\begin{smallmatrix}\mathbf{0} && I_N \\ I_N && \mathbf{0}\end{smallmatrix}\right]$.
\end{lmm}\vskip 1mm
\begin{proof} See~\cite{Lessard2016}, adapted for problems in the form (\ref{eq:Jglobal}). \end{proof}

\section{Generalization of First-Order Consensus Algorithms to Arbitrary Subspace Constraints}\label{sec:general}
The subspace defined by $\mathcal{U}$ has a corresponding symmetric projection matrix $\mathcal{P_U} \coloneqq \mathcal{U}(\mathcal{U}^T\mathcal{U})^{-1}\mathcal{U}^T$. The gradient projection method~\cite{Bertsekas1999} can solve (\ref{eq:Jglobal}) iteratively:
\begin{equation}\label{eq:gradientprojection}
\omega^{t+1}=\mathcal{P_U}\left(\omega^t-\mu \;\mathrm{col}\{\nabla J_k(\omega_k^t)\}_{k=1}^N\right).
\vspace{-2mm}\end{equation}
However, this is a centralized method due to its reliance on a centralized projection operation. To be implementable in a distributed setting, the projection operation must be replaced by a diffusion/mixing step using some symmetric gossip matrix $\mathcal{A}$. Matrix $\mathcal{A}$ must satisfy the sparsity pattern of the network (i.e., $\mathcal{A}_{ij}=0$ if $\{i,j\}\notin \mathcal{E}$) and the  convergence condition
\begin{equation}\label{eq:Acondition}
\underset{t\rightarrow \infty}{\mathrm{lim}}\mathcal{A}^t=\mathcal{P_U}.
\end{equation}
\begin{lmm}[\cite{DILorenzo2020}]\label{lem:Aiff}
Condition (\ref{eq:Acondition}) holds if and only if
\begin{equation}
\label{eq:Aiff}
\mathcal{A}\mathcal{P_U}=\mathcal{P_U}, \quad \mathcal{P_U}\mathcal{A}=\mathcal{P_U}, \quad \norm{\mathcal{A}-\mathcal{P_U}}<1.
\end{equation}
\end{lmm}
\vskip 2mm

The consensus algorithms address the problem of finding a common $\bar{\omega} \in \mathbb{R}^d$ amongst agents that minimizes $\sum_{k=1}^NJ_k(\bar{\omega})$. This problem can be equivalently expressed as
\vspace{-2mm}\begin{equation}\label{eq:consensus}
\mathrm{min}\;\sum_{k=1}^NJ_k(\omega_k)\; \mathrm{subject\;to}\; \omega_j=\omega_k \mathrm{\;for\;all\;}j,k \in \mathcal{V},
\end{equation} 
which is a special case of (\ref{eq:Jglobal}) with $\mathcal{U}=\mathbf{1}_N \otimes I_d$.

Distributed consensus algorithms utilize a doubly stochastic gossip matrix $W$ that satisfies the sparsity pattern of the network. These algorithms' fixed points must satisfy the conditions $\omega^*=W\omega^*$ and $\omega^* \in \mathcal{R}(\mathbf{1}_N)$, which are equivalent to the consensus condition $\mathcal{N}(I-W)=\mathcal{R}(\mathbf{1}_N)$. Enforcing double stochasticity is simply an application of Lemma \ref{lem:Aiff}, where $\mathcal{P_U}=\frac{1}{N}\mathbf{1}_N\mathbf{1}_N^T$. The consensus condition is a special case of the more general condition $\mathcal{N}(I-\mathcal{A})=\mathcal{R}(\mathcal{U})$, which implies that $\omega^*=\mathcal{A}\omega^*$ and the fixed point $\omega^*$ of the algorithm satisfies the subspace constraint $\omega \in \mathcal{R}(\mathcal{U})$. 

The resulting conclusion from the above discussion is that distributed first-order algorithms designed to accelerate DGD or remove its fixed point bias can be generalized to a larger class of problems by careful adjustment of the gossiping scheme, i.e., by replacing  $W$ with an $\mathcal{A}$ satisfying Lemma~\ref{lem:Aiff}. Table~\ref{table:algs} lists the generalized versions of select algorithms, presented using the notation  in Section~\ref{sec:algform}. For a given network topology, the choice of $\mathcal{U}$ impacts the performance of these algorithms, which can be shown by applying the upcoming analysis results.

\begin{table*}[]
\centering
\caption{DiSPO, DAS, and First-Order Distributed Optimization Algorithms Generalized to Arbitrary Subspace Constraints}
\resizebox{2\columnwidth}{!}{
\begin{threeparttable}
\begin{tabular}{c|c|c}
\textbf{Algorithm} & DiSPO\tnote{1}\quad\cite{Lorenzo2019} & DAS\tnote{1}\quad\cite{Nassif2020a,Nassif2020b}\\ 
$\left[\begin{array}{c|c}A & B \\ \hline C_y & C_\omega \\ \hline F_\xi & F_u\end{array}\right]$ 
& $\left[\begin{array}{c|c}\mathcal{A} & -\mu I_N \\ \hline I_N & I_N \\ \hline \mathbf{0} & \mathbf{0}\end{array}\right]$ 
& $\left[\begin{array}{c|c}\mathcal{A} & -\mu \mathcal{A} \\ \hline I_N & I_N \\ \hline \mathbf{0} & \mathbf{0}\end{array}\right]$ 
\\ \hline

\rule{0mm}{3mm}EXTRA\tnote{2}\quad\cite{Shi2015} & NIDS\tnote{2,3}\quad\cite{Li2019} & ED\tnote{2}\quad\cite{Yuan2019,Yuan2019a} \\
$\left[\begin{array}{c|c}\begin{matrix}I_N+\mathcal{A} & -\widetilde{\mathcal{A}} & \mu I_N \\ I_N & \mathbf{0} & \mathbf{0} \\ \mathbf{0} & \mathbf{0} & \mathbf{0}\end{matrix} & \begin{matrix}-\mu I_N \\ \mathbf{0} \\ I_N \end{matrix} \\ \hline \begin{matrix}I_N & \mathbf{0} & \mathbf{0}\end{matrix} & \begin{matrix}I_N & \mathbf{0} & \mathbf{0}\end{matrix} \\ \hline \rule{0mm}{3mm}\begin{matrix}\mathcal{U}^T & -\mathcal{U}^T & \mu \mathcal{U}^T\end{matrix} & \mathbf{0}\end{array}\right]$
& $\left[\begin{array}{c|c}\begin{matrix}I_N+\mathcal{A} & -\widetilde{\mathcal{A}} & \mu \widetilde{\mathcal{A}} \\ I_N & \mathbf{0} & \mathbf{0} \\ \mathbf{0} & \mathbf{0} & \mathbf{0}\end{matrix} & \begin{matrix}-\mu \widetilde{\mathcal{A}} \\ \mathbf{0} \\ I_N \end{matrix} \\ \hline \begin{matrix}I_N & \mathbf{0} & \mathbf{0}\end{matrix} & \begin{matrix}I_N & \mathbf{0} & \mathbf{0}\end{matrix} \\ \hline \rule{0mm}{3mm}\begin{matrix}\mathcal{U}^T & -\mathcal{U}^T & \mu \mathcal{U}^T\end{matrix} & \mathbf{0}\end{array}\right]$ 
& $\left[\begin{array}{c|c}\begin{matrix}2\widetilde{\mathcal{A}} & -\widetilde{\mathcal{A}} \\ I_N & \mathbf{0}\end{matrix} & \begin{matrix}-\mu \widetilde{\mathcal{A}} \\ -\mu I_N\end{matrix} \\ \hline \begin{matrix}I_N & \mathbf{0} \end{matrix} & \begin{matrix}I_N & \mathbf{0} \end{matrix} \\ \hline \rule{0mm}{3mm}\begin{matrix}\mathcal{U}^T & -\mathcal{U}^T\end{matrix} & \mathbf{0}\end{array}\right]$ 
\\ \hline

\rule{0mm}{3mm}DIGing\tnote{4}\quad\cite{Nedic2017} & uDIG\tnote{5}\quad\cite{Jakovetic2019} & uEXTRA\tnote{5}\quad\cite{Jakovetic2019} \\
$\left[\begin{array}{c|c}\begin{matrix}\mathcal{A} & -\mu I_N & \mathbf{0}\\ \mathbf{0} & \mathcal{A} & -I_N \\ \mathbf{0} & \mathbf{0} & \mathbf{0}\end{matrix} & \begin{matrix}\mathbf{0} \\ I_N \\I_N\end{matrix} \\ \hline \begin{matrix}\mathcal{A} & -\mu I_N & \mathbf{0}\end{matrix} & \begin{matrix}I_N & \mathbf{0} & \mathbf{0}\end{matrix} \\ \hline \rule{0mm}{3mm}\begin{matrix}\mathbf{0} & \mathcal{U}^T & -\mathcal{U}^T\end{matrix} & \mathbf{0}\end{array}\right]$
&$\left[\begin{array}{c|c}\begin{matrix}\mathcal{A} & -\mu I_N \\ (I_N-\mathcal{A})(\frac{L+m}{2}) & \mathcal{A}\end{matrix} & \begin{matrix}-\mu I_N \\ \mathcal{A}-I_N\end{matrix} \\ \hline \begin{matrix}I_N & \mathbf{0} \end{matrix} & \begin{matrix}I_N & \mathbf{0}\end{matrix} \\ \hline \rule{0mm}{3mm}\begin{matrix}\mathbf{0} & \mathcal{U}^T\end{matrix} & \mathbf{0}\end{array}\right]$
&$\left[\begin{array}{c|c}\begin{matrix}\mathcal{A} & -\mu I_N \\ (I_N-\mathcal{A})(L\mathcal{A}) & \mathcal{A}\end{matrix} & \begin{matrix}-\mu I_N \\ \mathcal{A}-I_N\end{matrix} \\ \hline \begin{matrix}I_N & \mathbf{0} \end{matrix} & \begin{matrix}I_N & \mathbf{0}\end{matrix} \\ \hline \rule{0mm}{3mm}\begin{matrix}\mathbf{0} & \mathcal{U}^T\end{matrix} & \mathbf{0}\end{array}\right]$ 
\\ \hline

\rule{0mm}{3mm}AugDGM\tnote{6}\quad\cite{Xu2015} & SVL\tnote{7}\quad\cite{Sundararajan2020} & ACC-DNGD-SC\tnote{8}\quad\cite{Qu2020} \\
$\left[\begin{array}{c|c}\begin{matrix}\mathcal{A} & -\mu I_N & \mathbf{0}\\ \mathbf{0} & \mathcal{A} & -\mathcal{A} \\ \mathbf{0} & \mathbf{0} & \mathbf{0}\end{matrix} & \begin{matrix}\mathbf{0} \\ \mathcal{A} \\I_N\end{matrix} \\ \hline \begin{matrix}\mathcal{A} & -\mu \mathcal{A} & \mathbf{0}\end{matrix} & \begin{matrix}I_N & \mathbf{0} & \mathbf{0}\end{matrix} \\ \hline \rule{0mm}{3mm}\begin{matrix}\mathbf{0} & \mathcal{U}^T & -\mathcal{U}^T\end{matrix} & \mathbf{0} \end{array}\right]$
& $\left[\begin{array}{c|c}\begin{matrix}\mathcal{A} & \beta I_N \\ \frac{\mathcal{A}-I_N}{\gamma} & I \end{matrix} & \begin{matrix}-\mu I_N \\ \mathbf{0}\end{matrix} \\ \hline \rule{0mm}{3.2mm}\begin{matrix}I_N-\frac{\delta}{\gamma}(I_N-\mathcal{A}) & \mathbf{0} \end{matrix} & \begin{matrix}I_N & \mathbf{0}\end{matrix} \\ \hline \rule{0mm}{3mm}\begin{matrix}\mathbf{0} & \mathcal{U}^T\end{matrix} & \mathbf{0}\end{array}\right]$
&$\left[\begin{array}{c|c}\begin{matrix}\frac{\mathcal{A}}{1+\alpha} & \frac{\alpha\mathcal{A}}{1+\alpha} & -\mu I_N & \mathbf{0} \\ \frac{\alpha\mathcal{A}}{1+\alpha} &\frac{\mathcal{A}}{1+\alpha} & -\frac{\mu}{\alpha} I_N & \mathbf{0} \\ \mathbf{0} & \mathbf{0} & \mathcal{A} & -I_N \\ \mathbf{0} & \mathbf{0} & \mathbf{0} & \mathbf{0} \end{matrix} & \begin{matrix}\mathbf{0} \\ \mathbf{0} \\ I_N \\I_N\end{matrix} \\ \hline \rule{0mm}{3.5mm}\begin{matrix}\frac{\alpha^2+1}{(1+\alpha)^2}\mathcal{A} & \frac{2\alpha}{(1+\alpha)^2}\mathcal{A} & \frac{-2\mu}{1+\alpha}I_N & \mathbf{0}\end{matrix} & \begin{matrix}I_N & \mathbf{0} & \mathbf{0} & \mathbf{0}\end{matrix} \\ \hline \rule{0mm}{3mm}\begin{matrix}\mathbf{0} & \mathbf{0} & \mathcal{U}^T & -\mathcal{U}^T\end{matrix} & \mathbf{0}\end{array}\right]$
\end{tabular}
\begin{tablenotes}[para]
\item[1] Fixed point bias $\mathcal{O}(\mu)$. 
\item[2] $\widetilde{\mathcal{A}}=\frac{1}{2}(I_N+\mathcal{A})$. 
\item[3] Supports non-smooth optimization.
\item[4] Designed for time-varying $\mathcal{A}^t$. \\
\item[5] All agents must know $L=\mathrm{max}\{L_k\}_{k=1}^N$ and $m=\mathrm{min}\{m_k\}_{k=1}^N$. \item[6] Supports uncoordinated step-sizes if replacing $\mu I_N$ with $\mathrm{diag}\{\mu_k\}_{k=1}^N$.\\
\item[7] Designed for optimal convergence rate in time-varying setting. Parameters $\mu,\beta,\gamma,\delta$ tuned according to Theorem 12 / Algorithm 2 in~\cite{Sundararajan2020}.\\
\item[8] Parameter $\alpha=\sqrt{m \mu}$. Designed such that linear convergence rate is less sensitive to increasing condition ratio $\kappa$.
\end{tablenotes}
\end{threeparttable}}
\label{table:algs}\vspace{-4mm}
\end{table*}
There is no guarantee that there will exist an $\mathcal{A}$ that satisfies some arbitrary network sparsity pattern. In the consensus case, the graph must simply be connected to guarantee existence. Equation (66) in \cite{Nassif2020b} defines a convex optimization problem based on the $\ell^1$ norm sparsity heuristic that takes an arbitrary sparsity pattern and generates an $\mathcal{A}$ that satisfies (\ref{eq:Aiff}) with minimal edges added to the original network topology.

\vspace{-1mm}\subsection{Numerical Example}\label{sec:num}
Consider a $4$-agent network with nonlinear, strongly convex local objective functions of the form $J_k(\omega_k)=a_k(\omega_k-b_k)^2-\cos(\omega_k)$, where $\{a_1,a_2,a_3,a_4\}=\{3,7,2,4\}$ and $\{b_1,b_2,b_3,b_4\}=\{-2,-1,5,12\}$. Distributed optimization is to be performed with a subspace constraint defined by  $\mathcal{U}=\left[\begin{smallmatrix}1 &\, 2 &\, 3 &\, 4\\1 &\, 1 &\, 2 &\, 2\end{smallmatrix}\right]^T$. We introduce gradient noise, where the noise bound $R$ is of the form $\sigma^2_wI$, with $\sigma_w=0.5$. Figure~\ref{fig:simulation} shows simulation results for the algorithms defined in Table~\ref{table:algs}, using a gossip matrix $\mathcal{A}_1$ that satisfies (\ref{eq:Aiff}) for the chosen $\mathcal{U}$, with step-size $\mu=0.012$. Despite the fact that $\omega^\mathrm{opt}=\left[\begin{smallmatrix}-0.719 & \, 3.996 & \, 3.277 & \, 7.991\end{smallmatrix}\right]^T$ (computed numerically using CVX~\cite{cvx}) is not a consensus solution, all generalized algorithms converge towards $\omega^\mathrm{opt}$ without bias. DAS and DiSPO converge with fixed point biases. 
\begin{figure}[h]
\includegraphics[width=\columnwidth]{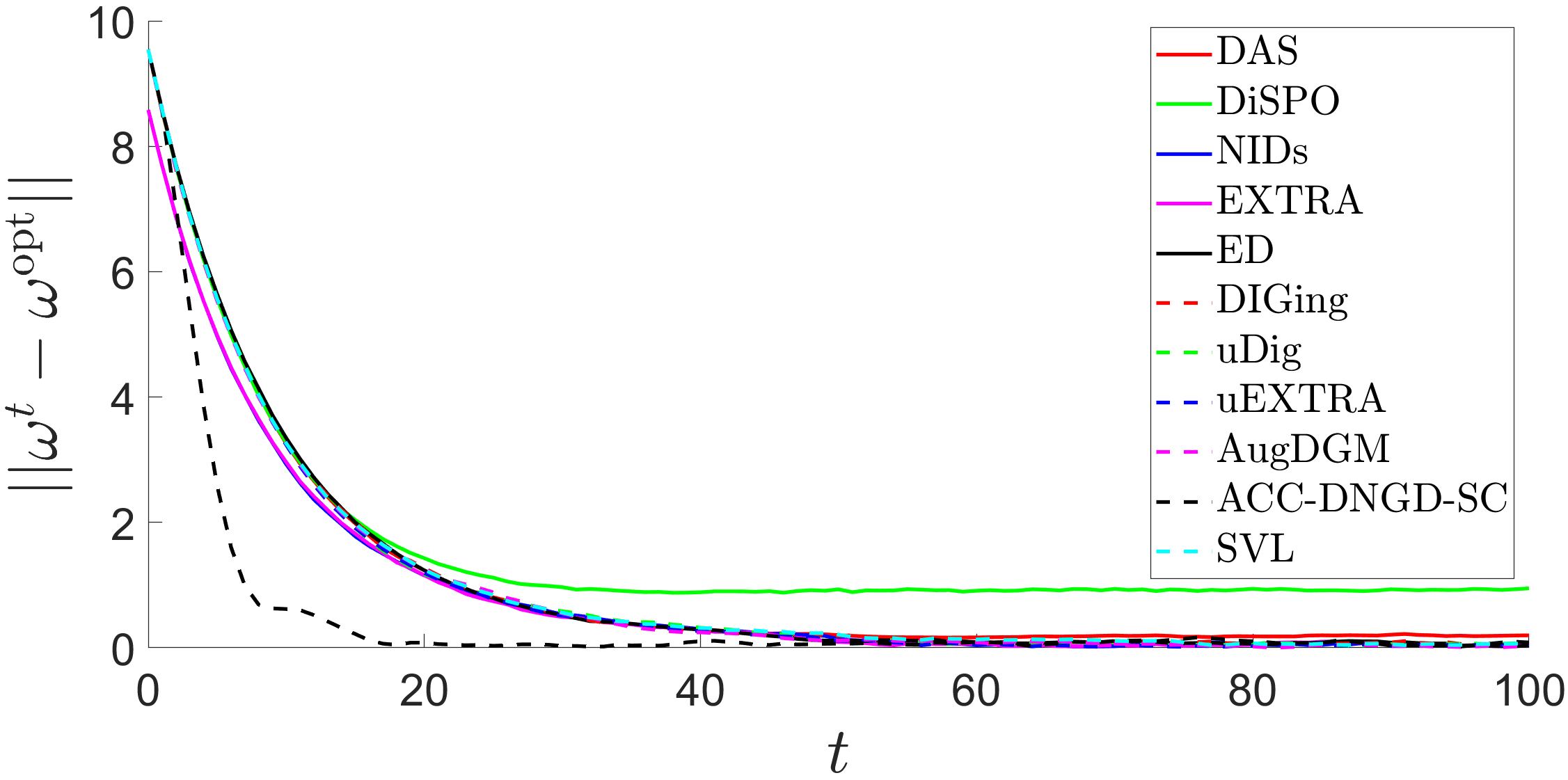}
\centering\vspace{-6mm}
\caption{Norm of error with respect to $\omega^\mathrm{opt}$ after performing optimization over a subspace constraint. All generalized consensus algorithms reach a steady state centered at zero error, while DAS and DiSPO reach a steady state centered about a nonzero bias.}
\label{fig:simulation}\vspace{-4mm}
\end{figure}
 
\section{Analysis Results}\label{sec:analysis}
The dynamics of the algorithms from (\ref{eq:algmodel}) and of $\Psi$ from (\ref{eq:psimodel}) can be used to define the following extended system $\hat{G}$:
\begin{equation}\label{eq:extendedsystem}
\hat{\xi}^{t+1}=\hat{A}\hat{\xi}^t+\hat{B}u^t+\begin{bmatrix}B \\ 0 \end{bmatrix}w^t,\quad z^t=\hat{C}\hat{\xi}^t+\hat{D}u^t,
\end{equation}
where $\hat{\xi}^t=\left[\begin{smallmatrix}\xi^t \\ \psi^t \end{smallmatrix}\right]$, $\hat{A}=\left[\begin{smallmatrix}A &\, 0 \\B^y_\Psi C_y &\, A_\Psi \end{smallmatrix}\right]$, $\hat{B}=\left[\begin{smallmatrix}B \\ B^u_\Psi\end{smallmatrix}\right]$, $\hat{C}=\left[\begin{smallmatrix}D^y_\Psi C_y &\, C_\Psi\end{smallmatrix}\right]$, and $\hat{D}=D^u_\Psi$. As before, in the absence of gradient noise, the extended system has fixed point $(\hat{\xi}^*,u^*,z^*)$.

Using this extended system, the following analysis results provide bounds on the rate of convergence and sensitivity.
\begin{thm}[Distributed Algorithm Rate of Convergence]\label{thm:converge}
Consider solving problem (\ref{eq:Jglobal}) for a set of local objective functions $J_k \in S(m_k,L_k)$ for all $k\in \mathcal{V}$, whose gradient computations satisfy the $\rho$-hard IQC defined by ($\Psi,M$) for a given $\rho>0$. Assume the noise $w^t \equiv 0$. Also, assume the algorithm has a unique fixed point and satisfies the invariant condition (\ref{eq:alginvariant}). Let $H$ be a matrix whose columns form a basis for the nullspace of $\begin{bmatrix}F_\xi & \mathbf{0} & F_u\end{bmatrix}$. 

If there exist $P \succeq 0$ and $\lambda \geq 0$ such that\vspace{-1mm}
 \begin{multline}\label{eq:LMIconverge} 
H^T\biggl(\begin{bmatrix}\hat{A}^TP\hat{A}-\rho^2P & \hat{A}^TP\hat{B}\\ \hat{B}^TP\hat{A} & \hat{B}^TP\hat{B} \end{bmatrix}\\ +\lambda \begin{bmatrix}\hat{C} & \hat{D}\end{bmatrix}^T M \begin{bmatrix}\hat{C} & \hat{D}\end{bmatrix}\biggr) H \preceq 0, 
 \end{multline} 
then $\norm{\xi^t-\xi^*}\leq c\rho^t\norm{\hat{\xi}^0-\hat{\xi}^*}$ for some constant $c>0$, for all $t\in \mathbb{N}$. \end{thm}
\begin{proof} 
Define the error states $\tilde{\xi}^t\coloneqq\hat{\xi}^t-\hat{\xi}^*$, $\tilde{u}^t\coloneqq u^t-u^*$, and $\tilde{z}^t\coloneqq z^t-z^*$. The columns of $H$ span the nullspace of $\begin{bmatrix}F_\xi & \mathbf{0} & F_u\end{bmatrix}$, so any vector $\left[\begin{smallmatrix}\tilde{\xi}^t \\ \tilde{u}^t\end{smallmatrix}\right]$ is of the form $Hh^t$ for some $h^t$. Pre- and post-multiply (\ref{eq:LMIconverge}) by $(h^t)^T$ and $h^t$, respectively, to obtain
\begin{equation*}(\tilde{\xi}^{t+1})^TP(\tilde{\xi}^{t+1})-\rho^2(\tilde{\xi}^t)^TP\tilde{\xi}^t+\lambda (\tilde{z}^t)^TM\tilde{z}^t \leq 0.\end{equation*}
 Multiply by $\rho^{-2t}$ and sum the resulting inequalities from $0$ to $T-1$ for any $T \in \mathbb{N}$. The first two terms produce a telescoping sum such that \begin{equation*}\rho^{-2T+2}(\tilde{\xi}^T)^TP(\tilde{\xi}^T)-\rho^2(\tilde{\xi}^0)^TP\tilde{\xi}^0+\lambda \sum_{i=0}^{T-1}\rho^{-2i}(\tilde{z}^i)^TM\tilde{z}^i \leq 0.\end{equation*}
 The summation is positive because the uncertainty satisfies a $\rho$-hard IQC. As a result, $(\tilde{\xi}^t)^TP(\tilde{\xi}^t)\leq \rho^{2t}(\tilde{\xi}^0)^TP\tilde{\xi}^0$, which implies that $\lambda_{\mathrm{min}}(P)\norm{\tilde{\xi}^t}^2\leq \rho^{2t}\lambda_{\mathrm{max}}(P)\norm{\tilde{\xi}^0}^2$, where $\lambda_{\mathrm{min}}(P)$ and $\lambda_{\mathrm{max}}(P)$ are the minimum and maximum eigenvalues of $P$, respectively.
Rearrange to produce the result: \vspace{-2mm}
 \begin{equation*}
\norm{\xi^t-\xi^*}\leq \norm{\hat{\xi}^t-\hat{\xi}^*}\leq c\rho^t\norm{\hat{\xi}^0-\hat{\xi}^*}, \quad \textstyle c=\sqrt{\frac{\lambda_{\mathrm{max}}(P)}{\lambda_{\mathrm{min}}(P)}}.
 \vspace{-2mm}\end{equation*}
\vspace{-1mm}\end{proof}
Multiple IQCs can be used simultaneously. For $r$ IQCs, the output $z^t$ in (\ref{eq:psimodelz}) becomes $\mathrm{col}\{z_i^t\}_{i=1}^r$. Additionally, the $\lambda M$ term in (\ref{eq:LMIconverge}) is replaced by the block-diagonal matrix formed from $\lambda_iM_i$ for $i=1,\ldots,r$, namely,
$\mathrm{blkdiag}\{\lambda_1M_1,\ldots,\lambda_rM_r\}$.
\begin{thm}[Distributed Algorithm Sensitivity]\label{thm:robust}
Consider solving problem (\ref{eq:Jglobal}) for a set of local objective functions $J_k \in S(m_k,L_k)$ for all $k\in \mathcal{V}$, whose gradient computations satisfy the $\rho$-hard IQC defined by ($\Psi,M$) for a given $\rho>0$. Assume the algorithm is subject to zero mean additive gradient noise satisfying Assumption~\ref{ass:noise}. Also, assume the algorithm has a unique fixed point and satisfies the invariant condition (\ref{eq:alginvariant}). Let $H$ be a matrix whose columns form a basis for the nullspace of $\begin{bmatrix}F_\xi & \mathbf{0} & F_u\end{bmatrix}$.

If there exist $P =\left[\begin{smallmatrix}P_{11} &\, P_{12} \\ P_{12}^T &\, P{22}\end{smallmatrix}\right]\succeq 0$ and $\lambda \geq 0$ such that
\begin{multline} \label{eq:LMIsensitivity} 
H^T\biggl(\begin{bmatrix}\hat{A}^TP\hat{A}-P & \hat{A}^TP\hat{B}\\ \hat{B}^TP\hat{A} & \hat{B}^TP\hat{B} \end{bmatrix}+\lambda \begin{bmatrix}\hat{C} & \hat{D}\end{bmatrix}^T M \begin{bmatrix}\hat{C} & \hat{D}\end{bmatrix}\\ +\begin{bmatrix}C_\omega^TC_\omega & 0 \\ 0 & 0 \end{bmatrix}\biggr) H \preceq 0,
\end{multline}
then $\gamma \leq \sqrt{\mathrm{tr}(RB^TP_{11}B)}$.
\end{thm}
\begin{proof}Define the error states $\tilde{\xi}^t\coloneqq\hat{\xi}^t-\hat{\xi}^*$, $\tilde{u}^t\coloneqq u^t-u^*$, $\tilde{\omega}^t\coloneqq \omega^t-\omega^*$ and $\tilde{z}^t\coloneqq z^t-z^*$. The columns of $H$ span the nullspace of $\begin{bmatrix}F_\xi & \mathbf{0} & F_u\end{bmatrix}$, so any vector $\left[\begin{smallmatrix}\tilde{\xi}^t \\ \tilde{u}^t\end{smallmatrix}\right]$ is of the form $Hh^t$ for some $h^t$. Pre- and post-multiply (\ref{eq:LMIsensitivity}) by $(h^t)^T$ and $h^t$, respectively, to obtain \begin{equation*}\begin{gathered} (\tilde{\xi}^{t+1})^TP(\tilde{\xi}^{t+1})-(\tilde{\xi}^t)^TP\tilde{\xi}^t-2(\tilde{\xi}^{t+1})^TP\left[\begin{smallmatrix}B\\0\end{smallmatrix}\right]w^t\\
+(w^t)^T\left[\begin{smallmatrix}B\\0\end{smallmatrix}\right]^TP\left[\begin{smallmatrix}B\\0\end{smallmatrix}\right]w^t+\lambda (\tilde{z}^t)^TM\tilde{z}^t + \norm{\tilde{\omega}^t}^2 \leq 0.\end{gathered}\end{equation*}
For the third term, substitute $\tilde{\xi}^{t+1}$ using (\ref{eq:extendedsystem}) to obtain \begin{equation*}\begin{gathered}(\tilde{\xi}^{t+1})^TP(\tilde{\xi}^{t+1})-(\tilde{\xi}^t)^TP\tilde{\xi}^t-2(\hat{A}\tilde{\xi}^t+\hat{B}\tilde{u}^t)^TP\left[\begin{smallmatrix}B\\0\end{smallmatrix}\right]w^t\\
-(w^t)^T\left[\begin{smallmatrix}B\\0\end{smallmatrix}\right]^TP\left[\begin{smallmatrix}B\\0\end{smallmatrix}\right]w^t+\lambda (\tilde{z}^t)^TM\tilde{z}^t + \norm{\tilde{\omega}^t}^2 \leq 0.\end{gathered}\end{equation*} 
Take the expectation. The third term is zero because $w^t$ is zero-mean and  $\xi^t$ is independent of $w^t$. Rearrange to obtain
\begin{equation*}\begin{gathered}\mathbb{E}(\tilde{\xi}^{t+1})^TP(\tilde{\xi}^{t+1})-\mathbb{E}(\tilde{\xi}^t)^TP\tilde{\xi}^t+\mathbb{E}\lambda (\tilde{z}^t)^TM\tilde{z}^t + \mathbb{E}\norm{\tilde{\omega}^t}^2\leq \\
\mathbb{E}(w^t)^T\left[\begin{smallmatrix}B\\0\end{smallmatrix}\right]^TP\left[\begin{smallmatrix}B\\0\end{smallmatrix}\right]w^t=\mathrm{tr}(RB^TP_{11}B).\end{gathered}\end{equation*}
Sum from $t=0$ to $t=T-1$ to obtain \vspace{-2mm}
\begin{equation*}\begin{gathered}\mathbb{E}\frac{1}{T}(\tilde{\xi}^T)^TP\tilde{\xi}^T-\mathbb{E}\frac{1}{T}(\tilde{\xi}^0)^TP\tilde{\xi}^0+\mathbb{E}\frac{1}{T}\sum_{t=0}^{T-1}\lambda (\tilde{z}^t)^TM\tilde{z}^t\\
+ \mathbb{E}\frac{1}{T}\sum_{t=0}^{T-1}\norm{\tilde{\omega}^t}^2 \leq\mathrm{tr}(RB^TP_{11}B).\end{gathered}\vspace{-2mm}\end{equation*}
 For a bounded initial condition, $(\tilde{\xi}^0)^TP\tilde{\xi}^0$ is bounded. The feasibility of (\ref{eq:LMIsensitivity}) implies that (\ref{eq:LMIconverge}) holds for $\rho=1$, which implies that $(\tilde{\xi}^T)^TP\tilde{\xi}^T$ is bounded. Take the limit as $T \rightarrow \infty$, noting that the third term is positive because a $\rho$-hard IQC satisfies the soft IQC condition. Substitute the definition of $\gamma$ to obtain\vspace{-4mm}
 \begin{equation*} 
 \gamma=\mathbb{E}\frac{1}{T}\sum_{t=0}^{T-1}\norm{\omega^t-\omega^*}^2\leq \mathrm{tr}(RB^TP_{11}B).
 \vspace{-4mm}\end{equation*}
 \vspace{-1mm}\end{proof}

Theorem~\ref{thm:converge} is simply \cite[Lemma~2]{Sundararajan2018}, but framed in terms of the more general IQC framework. In \cite{Sundararajan2018}, the differences $u^t-u^*$ and $y^t-y^*$ are assumed to satisfy a quadratic inequality, which is equivalent to having a static multiplier $\Psi$, for example, the distributed sector IQC from Lemma~\ref{lem:sector}. Allowing for dynamic characterizations of the uncertainty $\Delta$, as in the distributed weighted off-by-one IQC from Lemma~\ref{lem:weighted}, necessitates use of the extended system (\ref{eq:extendedsystem}) and reduces conservatism. A minimum $\rho$ can be found by performing a bisection search, checking for feasibility of (\ref{eq:LMIconverge}) iteratively. Theorem~\ref{thm:robust} is a new result, inspired by the robustness analysis of centralized algorithms in \cite{VanScoy2021}. The inequality (\ref{eq:LMIsensitivity}) is a linear matrix inequality (LMI) in $P$, so the minimum $\mathrm{tr}(RB^TP_{11}B)$ can be found directly by solving a  SDP.

A similar convergence result is presented in \cite{Sundararajan2020}, which assesses an algorithm's robustness with respect to time-varying networks. Comparatively, Theorem~\ref{thm:converge} has the potential to be less conservative since it considers a specific network rather than the worst-case over all networks that are bounded by a given spectral gap $\sigma$. Additionally, Theorem~\ref{thm:converge} and Theorem~\ref{thm:robust} support agent-specific parameters (e.g., $L_k$, $m_k$, $\mu_k$) rather than assuming these parameters to be uniform across the network's agents. The trade-off is that both LMIs scale with the size of the network, leading to increased computation time for analysis.

As mentioned in Section~\ref{sec:intro}, LMI-based analysis is beneficial because it does not rely on algorithm-specific expertise to obtain convergence and robustness guarantees. For example, our framework can validate the robustness bounds of ED from \cite{Yuan2020} without their sophisticated proof that uses the mean-value theorem to bound the error dynamics of the algorithm. Our approach also has the potential to provide numerical guarantees that are less conservative than those provided by alternative proofs, similar to how \cite{Lessard2016} finds an improved bound for Nesterov's Accelerated Method. Furthermore, expert analysis of the aforementioned algorithms is mainly restricted to a deterministic setting, and so our approach can provide robustness guarantees that were previously nonexistent. 

\section{Case Study}\label{sec:bias}
The above analysis tools are used to show how a generalized distributed consensus algorithm performs when solving a multitask inference problem found in \cite{Nassif2020a,Nassif2020b}. Here, the agents use streaming data to minimize their individual costs, where the optimal parameter must lie in a low-dimensional subspace. In this setting, the local objective functions can be considered to be the expectation of some loss function $Q(\omega_k;\mathbf{x}_k)$. Random variable $\mathbf{x}_k$ corresponds to data received by agent $k$, whose distribution is unknown. Since only a finite number of samples $\mathbf{x}_k$ are received by each agent, their local gradient computations are subject to gradient noise. Local costs are assumed twice differentiable and convex with bounded Hessian, which satisfies Assumptions~\ref{ass:Lip} and \ref{ass:Strong}. 
In \cite{Nassif2020a}, DAS, which is subject to a fixed point bias, is proposed to solve this type of problem. 

Consider the example from Section~\ref{sec:num}, where instead of prescribing local objective functions, worst-case analysis is performed over all objective functions in $\mathcal{S}(m_k,L_k)$ for all $k\in\mathcal{V}$. Two gossip matrices $\mathcal{A}_1$ and $\mathcal{A}_2$ are considered, with spectral gaps of $\sigma=0.19$ and $\sigma=0.63$, respectively. Figure~\ref{fig:multitask} shows the trade-off between  sensitivity $\gamma$ and  convergence rate $\rho$ as step-size $\mu$ varies for the DAS algorithm and the generalized AugDGM algorithm. These numerical results were obtained by implementing Theorems~\ref{thm:converge} and \ref{thm:robust}, subject to both IQCs defined in Lemmas~\ref{lem:sector} and \ref{lem:weighted}, in MATLAB with the CVX modeling language and the MOSEK~solver~\cite{mosek}. 

In general, an algorithm will exhibit faster convergence (smaller $\rho$) as $\mu$ increases at the cost of worse robustness (higher $\gamma$). At an algorithm-specific limiting $\mu$, which is dictated by $L$, $m$, and $\sigma$, any further increase of $\mu$ will be detrimental to both convergence rate and robustness. This phenomenon is shown in the AugDGM curve for $\sigma=0.63$, where worst-case convergence rate cannot improve beyond about $0.8$ (corresponding to $\mu=0.05$).
\begin{figure}[h]
\includegraphics[width=\columnwidth]{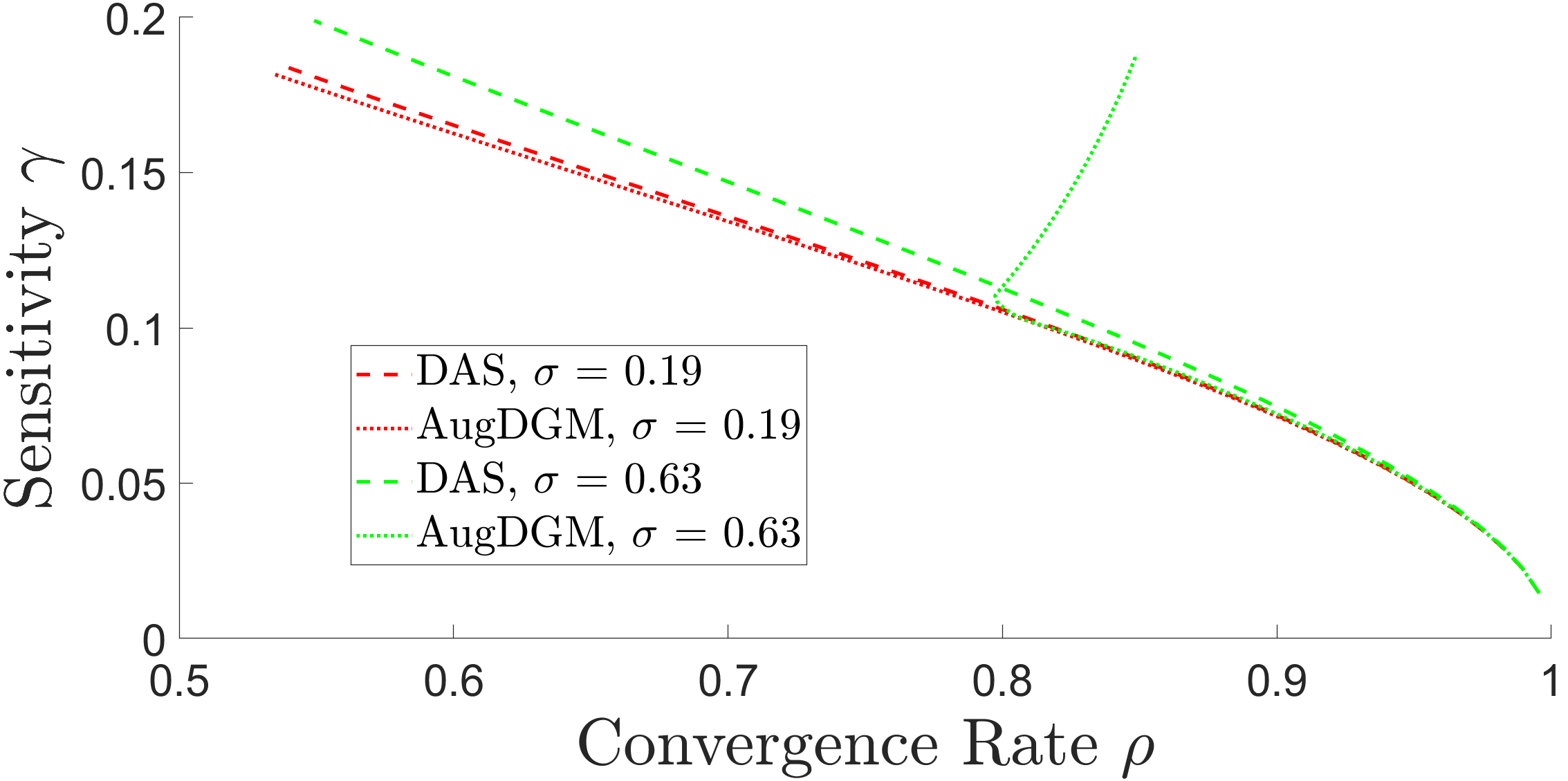}
\centering\vspace{-6mm}
\caption{Trade-off between sensitivity $\gamma$ and convergence rate $\rho$, obtained by varying algorithm step-size $\mu$, for an example multitask inference problem.}
\label{fig:multitask}
\end{figure}

For the $\sigma=0.19$ case, the AugDGM algorithm performs strictly better than DAS over the prescribed step-size range, with slightly lower $\gamma$ for each given $\rho$. Results for the $\sigma=0.63$ case are similar up to the limiting step-size of AugDGM. If prioritizing convergence rate, it would appear that DAS has greater potential in this scenario since it can achieve convergence rates between $0.5$ and $0.6$. However, both DAS and DiSPO are subject to fixed point biases, which is not captured in this analysis, since we bound variance about the algorithm's fixed point. In this large step-size regime, the bias introduced by DAS is large enough to prohibit the use of the algorithm. For example, if using the objective functions described in Section~\ref{sec:num} and gossip matrix $\mathcal{A}_1$, the bias $\norm{\omega^{\mathrm{opt}}-\omega^*}$ is $0.7$ at $\mu=0.05$, increasing to $1.5$ at $\mu=0.12$. Both are an order of magnitude larger than $\gamma$. For $\mathcal{A}_2$, the bias is even higher, at $1.8$ and $3.1$ for $\mu=0.05$ and $\mu=0.12$, respectively. In summary, applying a generalized distributed consensus algorithm designed for bias removal can achieve similar or better performance compared to the DAS algorithm, with a stronger benefit when the network is well-connected (low $\sigma$).

\section{Conclusion}\label{sec:conclude}
In this paper, we show that algorithms that remove the fixed point bias of DGD can be extended to solve the more general problem of distributed optimization over subspace constraints. We provide an analysis framework that can analyze the performance of these generalized algorithms in terms of worst-case robustness and convergence rate. Our framework can certify (or improve) the convergence rates provided by algorithm designers, as well as provide new robustness guarantees for algorithms that have not been previously considered in the stochastic setting. Finally, we demonstrate the utility of our framework by showing how a generalized consensus algorithm can be applied to a multitask inference problem.

\bibliographystyle{ieeetr}
\bibliography{references}

\begin{thebibliography}{10}

\bibitem{Yuan2016}
K.~Yuan, Q.~Ling, and W.~Yin, ``{On the Convergence of Decentralized Gradient
  Descent},'' {\em SIOPT}, vol.~26, no.~3, pp.~1835--1854, 2016.

\bibitem{Shi2015}
W.~Shi, Q.~Ling, G.~Wu, and W.~Yin, ``{EXTRA: An Exact First-Order Algorithm
  for Decentralized Consensus Optimization},'' {\em SIOPT}, vol.~25, no.~2,
  pp.~944--966, 2015.

\bibitem{Li2019}
Z.~Li, W.~Shi, and M.~Yan, ``{A Decentralized Proximal-Gradient Method With
  Network Independent Step-Sizes and Separated Convergence Rates},'' {\em IEEE
  Trans. Signal Process.}, vol.~67, pp.~4494--4506, 2019.

\bibitem{Yuan2019}
K.~Yuan, B.~Ying, X.~Zhao, and A.~H. Sayed, ``{Exact Diffusion for Distributed
  Optimization and Learning Part I: Algorithm Development},'' {\em IEEE Trans.
  on Signal Processing}, vol.~67, no.~3, pp.~708--723, 2019.

\bibitem{Nedic2017}
A.~Nedi{\'{c}}, A.~Olshevsky, and W.~Shi, ``{Achieving Geometric Convergence
  for Distributed Optimization Over Time-Varying Graphs},'' {\em SIOPT},
  vol.~27, no.~4, pp.~2597--2633, 2017.

\bibitem{Jakovetic2019}
D.~Jakovetic, ``{A Unification and Generalization of Exact Distributed
  First-Order Methods},'' {\em IEEE Trans. on Signal and Information Processing
  over Networks}, vol.~5, no.~1, pp.~31--46, 2019.

\bibitem{Xu2015}
J.~Xu, S.~Zhu, Y.~C. Soh, and L.~Xie, ``{Augmented Distributed Gradient Methods
  for Multi-Agent Optimization Under Uncoordinated Constant Stepsizes},'' in
  {\em CDC}, vol.~54, pp.~2055--2060, IEEE, 2015.

\bibitem{Sundararajan2020}
A.~Sundararajan, B.~Van~Scoy, and L.~Lessard, ``{Analysis and Design of
  First-Order Distributed Optimization Algorithms Over Time-Varying Graphs},''
  {\em IEEE Trans. on Control of Network Systems}, vol.~7, no.~4,
  pp.~1597--1608, 2020.

\bibitem{Nesterov2018}
Y.~Nesterov, {\em {Lectures on Convex Optimization}}, vol.~137 of {\em Springer
  Optimization and Its Applications}.
\newblock Springer, 2018.

\bibitem{Qu2020}
G.~Qu and N.~Li, ``{Accelerated Distributed Nesterov Gradient Descent},'' {\em
  IEEE Trans. on Autom. Control}, vol.~65, no.~6, pp.~2566--2581, 2020.

\bibitem{Yuan2020}
K.~Yuan, S.~A. Alghunaim, B.~Ying, and A.~H. Sayed, ``{On the Influence of
  Bias-Correction on Distributed Stochastic Optimization},'' {\em IEEE Trans.
  on Signal Processing}, vol.~68, pp.~4352--4367, 2020.

\bibitem{Nassif2020a}
R.~Nassif, S.~Vlaski, and A.~H. Sayed, ``{Adaptation and Learning Over Networks
  Under Subspace Constraints Part I: Stability Analysis},'' {\em IEEE Trans. on
  Signal Processing}, vol.~68, pp.~1346--1360, 2020.

\bibitem{Lorenzo2019}
P.~D. Lorenzo, S.~Barbarossa, and S.~Sardellitti, ``{Distributed Signal
  Recovery Based on In-Network Subspace Projections},'' in {\em ICASSP},
  pp.~5242--5246, IEEE, 2019.

\bibitem{Nassif2020b}
R.~Nassif, S.~Vlaski, and A.~H. Sayed, ``{Adaptation and Learning Over Networks
  Under Subspace Constraints Part II: Performance Analysis},'' {\em IEEE Trans.
  on Signal Processing}, vol.~68, no.~1, pp.~1--1, 2020.

\bibitem{Sayed2014}
A.~Sayed, ``{Adaptation, Learning, and Optimization over Networks},'' {\em
  Found. Trends Mach. Learn.}, vol.~7, no.~4-5, pp.~311--801, 2014.

\bibitem{Megretski1997SystemConstraints}
A.~Megretski and A.~Rantzer, ``{System Analysis via Integral Quadratic
  Constraints},'' {\em IEEE Trans. on Autom. Control}, vol.~42, no.~6,
  pp.~819--830, 1997.

\bibitem{Jaoude2019}
D.~A. Jaoude, D.~Muniraj, and M.~Farhood, ``{Robustness Analysis of Uncertain
  Time-Varying Interconnected Systems Using Integral Quadratic Constraints},''
  in {\em American Control Conference}, IEEE, 2019.

\bibitem{Fry2021}
J.~M. Fry, D.~Abou~Jaoude, and M.~Farhood, ``{Robustness analysis of uncertain
  time-varying systems using integral quadratic constraints with time-varying
  multipliers},'' {\em International Journal of Robust and Nonlinear Control},
  vol.~31, no.~3, pp.~733--758, 2021.

\bibitem{Lessard2016}
L.~Lessard, B.~Recht, and A.~Packard, ``{Analysis and Design of Optimization
  Algorithms via Integral Quadratic Constraints},'' {\em SIOPT}, vol.~26,
  no.~1, pp.~57--95, 2016.

\bibitem{VonScoy2018}
B.~Van~Scoy, R.~A. Freeman, and K.~M. Lynch, ``{The Fastest Known Globally
  Convergent First-Order Method for Minimizing Strongly Convex Functions},''
  {\em IEEE Control Syst. Lett.}, vol.~2, pp.~49--54, 2018.

\bibitem{Cyrus2018}
S.~Cyrus, B.~Hu, B.~Van~Scoy, and L.~Lessard, ``{A Robust Accelerated
  Optimization Algorithm for Strongly Convex Functions},'' in {\em American
  Control Conference}, pp.~1376--1381, IEEE, 2018.

\bibitem{Sundararajan2018}
A.~Sundararajan, B.~Hu, and L.~Lessard, ``{Robust Convergence Analysis of
  Distributed Optimization Algorithms},'' in {\em Allerton Conference on
  Communication, Control, \& Computing}, vol.~55, pp.~1206--1212, 2017.

\bibitem{VanScoy2021}
B.~Van~Scoy and L.~Lessard, ``{The Speed-Robustness Trade-Off for First-Order
  Methods with Additive Gradient Noise},'' 2021, arXiv:2109.05059v1.

\bibitem{Bertsekas1999}
D.~Bertsekas, ``{Nonlinear Programming},'' 1999.

\bibitem{DILorenzo2020}
P.~Di~Lorenzo, S.~Barbarossa, and S.~Sardellitti, ``{Distributed Signal
  Processing and Optimization Based on In-Network Subspace Projections},'' {\em
  IEEE Trans. Signal Process.}, vol.~68, pp.~2061--2076, 2020.

\bibitem{Yuan2019a}
K.~Yuan, B.~Ying, X.~Zhao, and A.~H. Sayed, ``{Exact Diffusion for Distributed
  Optimization and Learning Part II: Convergence Analysis},'' {\em IEEE Trans.
  on Signal Processing}, vol.~67, no.~3, pp.~724--739, 2019.

\bibitem{cvx}
M.~C. Grant and S.~P. Boyd, ``{CVX: Matlab Software for Disciplined Convex
  Programming},'' 2012.

\bibitem{mosek}
M.~ApS, ``{The MOSEK optimization toolbox for MATLAB manual. Version 9.1.},''
  2019.

\end{thebibliography}

%

\end{document}